\documentclass[12pt]{amsart}

\usepackage{amsmath}
\usepackage{amssymb}
\usepackage{amsfonts}
\usepackage{amsthm}
\usepackage{graphicx}

\DeclareMathOperator{\Tr}{Tr}

\newtheorem{theorem}{Theorem}
\newtheorem{prop}{Proposition}
\newtheorem{lemma}{Lemma}

\theoremstyle{definition}
\newtheorem{defi}[theorem]{Definition}

\newtheorem{remark}[theorem]{Remark}

\newcommand{\cW}{\mathcal{W}_q}
\title{ A Projective Representation of the Modular Group}
\author{Nadav Kohen}
\author{Charles Frohman}                             

\address{ Department of Mathematics \\ The University of Iowa}
\email{nadav-kohen@uiowa.edu, charles-frohman@uiowa.edu}

\begin{document}                              

\maketitle

\begin{abstract}
Quantum Teichmuller theory assigns invariants to three-manifolds via projective representations of mapping class groups derived from the representation of a noncommutative torus. Here, we focus on a representation  of the simplest  non-commutative torus which remains fixed by all elements of the mapping class group of the torus, $SL_2(\mathbb{Z})$. Also known as the modular group.  We use this representation to associate a matrix to each element of $SL_2(\mathbb{Z})$; we then compute the trace and determinant of the associated matrix. 
\end{abstract}

\section{Introduction}

Fock and Chekhov \cite{FC} defined a noncommutative algebra related to the Teichmuller space of a punctured surface. The algebra is a noncommutative torus whose defining relations come from a triangulation of the underlying surface.  There is an action of the mapping class group on this algebra, and if a representation is fixed by this action, then this gives rise to a projective representation of the mapping class group.

In this paper, we study a toy model of quantum Teichmuller space, the noncommutative torus in two variables, $\cW$, which can be thought of as being associated to the two torus $T^2=S^1\times S^1$.
 The ``trivial'' represetantation of $\cW$ is fixed by all elements of the mapping class group, and so gives rise to a projective representation of the mapping class group of the torus, $SL_2(\mathbb{Z})$. In this paper we compute the trace, and the determinant of a matrix associated to each element of $SL_2(\mathbb{Z})$ by this projective representation in the case of $q$ being a root of unity of prime order.

 We begin with a section of preliminaries, starting with the definition of $\cW$ and a description of the action of $SL_2(\mathbb{Z})$ on $\cW$ as automorphisms.
 After reviewing properties of matrix algebras in subsection \ref{mat}, we will review  representations of algebras in subsection \ref{rep}. We finish the preliminaries with a review of Gauss sums.

 In section \ref{reptor} we give models of the irreducible representations of $\cW$ and prove they are indeed irreducible.
 
 Next, in section \ref{skno}, we take an arbitrary element $B\in SL_2(\mathbb{Z})$ and find a matrix whose action by conjugation is the same as $B$'s action as an automorphism on the representation of $\cW$. Finally in sections \ref{trcalc} and \ref{detcalc}, we calculate the trace and determinant (respectively) of the conjugating matrix.

\section{Preliminaries}
\subsection{Noncomutative Tori}\label{weyl}
Let $\cW = \mathbb{C}[l,l^{-1},m,m^{-1}]_q$ be the non-commuting Laurent polynomials in variables $l$ and $m$ with $lm = q^2ml$ for some $q\in\mathbb{C}\backslash \{0\}$. We will study $\cW$ using the following basis. Let $e_{r,s} = q^{-rs}l^rm^s$. The set $\{e_{r,s}\mid r,s\in\mathbb{Z}\}$ forms a basis for $\cW$ over $\mathbb{C}$ where we can take products of elements in this basis as follows,
\begin{align*} e_{p,t}*e_{r,s} &= q^{-pt-rs}l^pm^tl^rm^s\\
&= q^{-pt-rs}l^p(q^{-2tr}l^rm^t)m^s\\
&= q^{-pt-rs-2tr}q^{(p+r)(t+s)}q^{-(p+r)(t+s)}l^{p+r}m^{t+s}\\
&= q^{ps-rt}e_{p+r,t+s}\\
&= q^{\begin{vmatrix}
p & t\\
r & s
\end{vmatrix}}e_{p+r,t+s}
\end{align*}

Lastly, note that if $n$ is odd and $q$ is a primitive $n$th root of unity, then the center of $\cW$ (see Definition \ref{algdef}) is spanned by $\{e_{np,nt}\mid p,t\in\mathbb{Z}\}$. We can now view $\cW$ as a module over its center (see Definition \ref{repdef}).

\begin{prop} The algebra $\cW$ is free module over its center with rank $n^2$ and basis $\{e_{i,j}\mid 0\leq i,j<n\}$.
\end{prop}
\begin{proof}
It is clear that $\{e_{i,j}\mid 0\leq i,j<n\}$ spans $\cW$ over $Z(\cW)$ since for any $r,s\in\mathbb{Z}$ we can find $p$ and $t$ so that $r = pn + i$ and $s = tn + j$ with $0\leq i,j < n$ so that $e_{r,s} = q^{-\begin{vmatrix}pn & tn\\ i & j\end{vmatrix}}e_{pn,tn}*e_{i,j}$. To show that this set is linearly independent, first note that $\{e_{np,nt}\mid p,t\in\mathbb{Z}\}$ is linearly independent over $\mathbb{C}$ as it is a subset of a basis for $\cW$ over $\mathbb{C}$. Suppose that $\sum_{i,j=0}^{n-1}\alpha_{i,j}e_{i,j} = 0$ with $\alpha_{i,j}\in Z(\cW) = \langle e_{np,nt}\mid p,t\in\mathbb{Z}\rangle$. Write $\alpha_{i,j} = \sum_k\beta_{i,j}^ke_{np_{i,j}^k,nt_{i,j}^k}$ where $\beta_{i,j}^k\in\mathbb{C}$ and the choice of the $\beta_{i,j}^k$ is unique. Then
\begin{align*}
0 &= \sum_{i,j=0}^{n-1}\alpha_{i,j}e_{i,j}\\
&=  \sum_{i,j=0}^{n-1}\left(\sum_k\beta_{i,j}^ke_{np_{i,j}^k,nt_{i,j}^k}\right)e_{i,j}\\
&= \sum_{i,j=0}^{n-1}\sum_k\beta_{i,j}^kq^{\begin{vmatrix}np_{i,j}^k & nt_{i,j}^k\\i & j\end{vmatrix}}e_{np_{i,j}^k + i,nt_{i,j}^k + j}\\
&= \sum_{i,j=0}^{n-1}\sum_k\beta_{i,j}^ke_{np_{i,j}^k + i,nt_{i,j}^k + j}
\end{align*}
But this is a linear combination of unique elements in $\{e_{i,j}\mid i,j\in\mathbb{Z}\}$, our basis for $\cW$ over $\mathbb{C}$, and thus $\forall i,j,k$ it must be that $\beta_{i,j}^k = 0$ by the linear independence of our basis. Therefore, $\forall i,j\in\{0,\ldots,n-1\}$, $\alpha_{i,j} = 0$.
\end{proof}

\begin{prop} There is a left action of  $SL_2(\mathbb{Z})$ as automorphisms on  $\cW$ defined by $$\begin{pmatrix}
a & b\\
c & d
\end{pmatrix}e_{p,t} = e_{ap + bt,cp + dt}.$$
\end{prop}
\begin{proof}
We need only show that $\begin{pmatrix}
a & b\\
c & d
\end{pmatrix}(e_{p,t}*e_{r,s}) = \begin{pmatrix}
a & b\\
c & d
\end{pmatrix}e_{p,t} * \begin{pmatrix}
a & b\\
c & d
\end{pmatrix}e_{r,s}$. This can be accomplished through direct computation:
\begin{align*}
\begin{pmatrix}
a & b\\
c & d
\end{pmatrix}(e_{p,t} * e_{r,s}) &= q^{\begin{vmatrix}
p & t\\
r & s
\end{vmatrix}}\begin{pmatrix}
a & b\\
c & d
\end{pmatrix}e_{p+r,t+s}\\
&= q^{\begin{vmatrix}
p & t\\
r & s
\end{vmatrix}} e_{a(p+r) + b(t+s),c(p+r) + d(t+s)}\\
&= q^{\begin{vmatrix}a & b\\c & d\end{vmatrix}\cdot\begin{vmatrix}p & r\\t & s\end{vmatrix}}e_{ap + ar + bt + bs, cp + cr + dt + ds}\\
&= q^{\begin{vmatrix}ap + bt & ar + bs\\cp + dt & cr + ds\end{vmatrix}}e_{ap + bt + ar + bs, cp + dt + cr + ds}\\
&= e_{ap + bt,cp + dt} * e_{ar + bs,cr + ds}\\
&= \begin{pmatrix}
a & b\\
c & d
\end{pmatrix}e_{p,t} * \begin{pmatrix}
a & b\\
c & d
\end{pmatrix}e_{r,s}
\end{align*}
\end{proof}

\subsection{Matrix Algebras} \label{mat}
\begin{defi}\label{algdef} An \textbf{algebra} $A$ over a field $F$ is a vector space with an additional bilinear binary operation $\cdot:A\times A\rightarrow A$ usually called multiplication. We assume that the multiplication is associative, and there is a unit element.  The \textbf{center} of $A$, $Z(A)$ are those elements of $A$ that commute with all other elements, i.e. $Z(A) = \{x\in A\mid x\cdot a = a\cdot x\ \forall a\in A\}$.
\end{defi}
\begin{defi}\label{repdef} Given an algebra, $A$, a \textbf{left $A$-module} or \textbf{representation} of $A$ is a vector space $V$ over $\mathbb{C}$ along with a homomorphism $\rho:A\rightarrow End(V)$. We say the representation is \textbf{irreducible} if $\rho$ is onto.
\end{defi}

Let $M_n(\mathbb{C})$ be the algebra of $n\times n$-matrices with complex entries. There is a standard basis denoted $E_{i,j}$ of matrices that are zero except in the $(i,j)$-entry. In this paper the indices $i$ and $j$ run from $0$ to $n-1$.  It is well known that
\begin{equation*} E_{i,j}E_{k,l}=\delta^j_kE_{i,l} \end{equation*} where $\delta^j_k$ is the Kronecker delta. Note that the center of $M_n(\mathbb{C})$ is exactly all scalar multiples of the identity. We will also use the Skolem-Noether \cite{G}  theorem, which ensures that every automorphism of $M_n(\mathbb{C})$ is inner, i.e, if $\theta:M_n(\mathbb{C})\rightarrow M_n(\mathbb{C})$ is an automorphism, then there exists $C\in M_n(\mathbb{C})$, unique up to scalar multiples, so that for all $A\in M_n(\mathbb{C})$,
\begin{equation*} \theta(A)=C^{-1}AC. \end{equation*}

\begin{remark} After choosing a basis for the vector space $V$, $End(V)$ can be identified with $M_n(\mathbb{C})$. In this paper we treat representations as homomorphisms into $M_n(\mathbb{C})$. \end{remark}

\subsection{Irreducible Representations}\label{rep}

We begin this section, by showing that irreducible representations of an algebra, $A$, are determined, up to equivalence, by their kernels. Then we show that under certain circumstances (that will appear in Section \ref{reptor}), those kernels are fully determined by their intersections with $Z(A)$.

\begin{prop}\label{uniqueirr} If $A$ is an associative  algebra, and $\rho:A\rightarrow M_n(\mathbb{C})$ is irreducible with $I = \ker\rho$, then $\rho$ is completely determined by $I$.
\end{prop}
\begin{proof} Suppose $\rho_1,\rho_2:A\rightarrow M_n(\mathbb{C})$ both onto, such that $\ker\rho_1=\ker\rho_2$. Then there exists a well defined endomorphism on $M_n(\mathbb{C})$, $\rho_2\circ\rho_1^{-1}$. Therefore, by the Skolem-Noether theorem, this map must be inner so that there is some $C\in GL_n(\mathbb{C})$ such that $C\rho_1C^{-1} = \rho_2$. 
\end{proof}

\begin{prop}\label{amax} Let $A$ be an associative algebra that is a free module of rank $n^2$ over its center. Let $\rho:A\rightarrow M_n(\mathbb{C})$ be an irreducible representation, and let $I=ker(\rho)$, then $(I\cap Z(A))\cdot A = I$.
In particular, this means the kernel of $\rho$ is determined by its intersection with the center of $A$. \end{prop}
\begin{proof}
Because $I$ is an ideal of $A$, it is clear that $(I\cap Z(A))\cdot A\subseteq I$, so it is sufficient show the other inclusion. Notice that $\rho\vert_{Z(A)}:Z(A)\rightarrow Z(M_n(\mathbb{C})) = \mathbb{C}I_n$ is a homomorphism from $Z(A)$ onto a field; consequently, $\ker(\rho\vert_{Z(A)}) = I\cap Z(A)$ is a maximal ideal of $Z(A)$. If $B=\{e_i\}_{i=1}^{n^2}$ is a basis for $A$ over $Z(A)$ (of order $n^2$), then $\rho(B)$ must be a spanning set of $M_n(\mathbb{C})$ since $\rho$ is onto. Therefore, the elements of $\rho(B)$ are linearly independent as they span an $n^2$-dimensional vector space. Let 
$$J_i:=\{z\in Z(A)\vert \exists ze_i + \sum_{j\neq i}z_je_j\in I\}.$$
This is an ideal of $Z(A)$ because the $e_i$ are a basis so that the expression $ze_i + \sum_{j\neq i}z_je_j$ is unique. Notice the fact that for all $i$, $I\cap Z(A)$ is contained in $J_i$ which implies $I\cap Z(A) = J_i$ since $I\cap Z(A)$ is maximal. As this is true for each $i$, we get that for any $i$ (since $J_{i_1} = J_{i_2}$), $I\subseteq J_i\cdot A$. Therefore, $(I\cap Z(A))\cdot A = I$.
\end{proof}

\subsection{ Legendre Symbols and Gauss Sums}
Lastly, we introduce Gauss Sums which will be instrumental in Section \ref{trcalc}.
\begin{defi}
Given two integers $a$ and $p$ such that $a\not\equiv 0\mod p$, then the associated \textbf{quadratic symbol}, or \textbf{Legendre symbol}, is defined to be
$$\left(\frac{a}{p}\right) = \left\{\begin{array}{ccc}
1 & \text{if} & a\equiv x^2\mod p\\
-1 & \text{if} & a\not\equiv x^2\mod p
\end{array}\right.$$ for any integer $x$. This value depends only on the residue class of $a\mod p$.
\end{defi}

\begin{defi} A \textbf{Gauss quadratic sum} is a sum of the form,
$$\sum_{x\mod b}e^{\frac{2\pi i}{b}ax^2} =: G(a,b)$$
where $a$ and $b$ are relatively prime, non-zero integers with $b>0$.
\end{defi}

Lang gives an exposition \cite{L} of Gauss' calculation showing that if $b\geq 1$ is odd, then $G(a,b) = \left(\frac{a}{b}\right)G(1,b)$, and that $G(1,b) = \frac{1 + i^{-b}}{1 + i^{-1}}\sqrt{b}$ so that for $b\geq 1$ odd, $$G(a,b) = \left(\frac{a}{b}\right)\frac{1 + i^{-b}}{1 + i^{-1}}\sqrt{b}.$$.

\section{Irreducible Representations of the Noncommutative Torus}\label{reptor}

We now describe representatives of every irreducible representation of $\cW$. Choose an $n$th root $b^{1/n}$ of $b$.  Define $\rho_{a,b}:\cW\rightarrow M_n(\mathbb{C})$ to be the representation of $\cW$ determined by $\rho_{a,b}(l) = L_a$ and $\rho_{a,b}(m) = M_b$ where $a,b\in\mathbb{C}$ and
$$L_a = \begin{pmatrix}
0 & 0 & \cdots & 0 & a\\
1 & 0 & \cdots & 0 & 0\\
0 & 1 & \cdots & 0 & 0\\
\vdots & \vdots & \ddots & \vdots & \vdots\\
0 & 0 & \cdots & 1 & 0
\end{pmatrix},\ M_b = \begin{pmatrix}
b^{\frac{1}{n}} & 0 & 0 & \cdots & 0\\
0 & b^{\frac{1}{n}}q^{-2} & 0 & \cdots & 0\\
0 & 0 & b^{\frac{1}{n}}q^{-4} & \cdots & 0\\
\vdots & \vdots & \vdots & \ddots & \vdots\\
0 & 0 & 0 & \cdots & b^{\frac{1}{n}}q^{-2(n-1)}
\end{pmatrix}$$

\begin{remark} Even though the matrix $M_b$ depends on the choice of $b^{1/n}$, the equivalence class of the representation does not. As this is a representation of an associative algebra that is free of rank $n^2$ over its center, by Proposition \ref{amax} the representation is determined by the intersection of its kernel with $Z(\cW)$. In this case that is the ideal $(l^n-a,m^n-b)$, which does not depend on the choice of $n$th root of $b$.\end{remark}

\begin{prop} The representation, $\rho_{a,b}:\cW\rightarrow M_n(\mathbb{C})$, is irreducible (surjective).
\end{prop}
\begin{proof}
Recall that $1 + q + q^2 + \cdots + q^{n-1} = \frac{q^n-1}{q-1} = 0$. We have that $\sum_{i=0}^{n-1}M_1^i = nE_{0,0}$, the matrix where every entry but the top left corner is zero.  Hence, $\frac{1}{n}\sum_{i=0}^{n-1}(b^{-\frac{1}{n}}M_b)^i = \frac{1}{n}\sum_{i=0}^{n-1}M_1^i = E_{0,0}$. Subsequently, $L_a^iE_{0,0} = E_{i,0}$ for $0\leq i\leq n-1$. Finally, $E_{i,0}(\frac{1}{a}L_a)L_a^{k-1} = E_{i,n-k}$ for $1\leq k\leq n-1$ so that $L_a,M_b$ together span every $E_{i,j}$ which spans all of $M_n(\mathbb{C})$. Therefore, $\rho_{a,b}$ is irreducible.
\end{proof}

\begin{prop}\label{fixedrep} If $B\in SL_2(\mathbb{Z})$, and $1$ is not an eigenvalue of $B$ the only representations of $\cW$ that can be fixed by $B$ are $\rho_{a,b}:\cW\rightarrow M_n(\mathbb{C})$ were $a$ and $b$ are roots of unity whose order divides the determinant of $B-I_2$. \end{prop}

\proof  Suppose that $B=\begin{pmatrix} a & b \\ c & d \end{pmatrix}$ and $\rho_{e,f}:\cW\rightarrow M_n(\mathbb{C})$ is fixed by $B$. Since irreducible representations of $\cW$ are determined by the intersection of their kernel with the center we are looking for $\lambda,\mu \in \mathbb{C}-\{0\}$ that solve
the system of equations 
\begin{equation*} \lambda^n=e , \ \mu^n=f \end{equation*} \begin{equation*} \lambda^{na}\mu^{nb}=e , \ \lambda^{nc}\mu^{nd}=f. \end{equation*}
Performing a multiplicative row operation we get,
\begin{equation*} \lambda^{n(a-1)}\mu^{nb}=1 , \ \lambda^{nc}\mu^{n(d-1)}=1.\end{equation*}
Using elementary operations of determinant $1$, done multiplicatively, we can make this system,
\begin{equation*} \lambda^{ne_1}=1, \ \mu^{ne_2}=1,\end{equation*} where $e_1e_2$ is equal to the determinant of $ \begin{pmatrix} a-1 & b \\ c & d-1 \end{pmatrix}$.
From this we see that $\lambda^n$ and $\mu^n$ are roots of unity whose order divides the determinant of $B-I_2$.

\qed

\begin{remark} For a particular $\begin{pmatrix} a & b \\c & d \end{pmatrix}$ there can be some fixed representations, that hold out promise of invariants of the mapping cylinders of those matrices.  In order to get a problem that we can solve uniformly we now restrict our attention to the representation $\rho_{1,1}:\cW\rightarrow M_n(\mathbb{C})$, which is fixed by all elements of $SL_2(\mathbb{Z})$. We are in fact studying a projective representation of the mapping class group, of the torus that should be related to the Witten-Reshetikhin-Turaev representations \cite{BWP}. \end{remark} 

\begin{remark}Proposition \ref{fixedrep} shows us that the only representation $\rho_{a,b}$ fixed by all of $SL_2(\mathbb{Z})$ is $\rho_{1,1}$. From now on, we will refer to $\rho_{1,1}$ as $\rho$, $L_1$ as $L$, and $M_1$ as $M$. \end{remark}

\section{Finding the Conjugating Matrix using the Skolem-Noether Theorem} \label{skno}
We now know that every $B = \begin{pmatrix}
a & b\\
c & d
\end{pmatrix}\in SL_2(\mathbb{Z})$ acts as automorphisms of $\cW$, and consequently on $Z(\cW)$. Thus, if any ideal $I\subseteq Z(\cW)$ is fixed by $B$, then $B$ induces an automorphism on $M_n(\mathbb{C})$ as inner automorphisms. Our goal in this section is to determine the conjugating matrix associated with the automorphism induced by $B$. We begin with the following crucial observation,
\begin{lemma}\label{conjcomp}
The first column of $CE_{k,0}C^{-1}$ is a constant multiple of the $k$th column of $C$.
\end{lemma}
\begin{proof} Note that
$$ME_{k,l}N = (m_{i,j})E_{k,l}(n_{i,j}) = (m_{i,j})(\delta_{j}^kn_{l,j}) = (m_{i,k}n_{l,j})$$
where $\delta_{j}^k = \left\{\begin{array}{cc}
0 & k\neq j\\
1 & k = j
\end{array}\right.$. Specifically, the first column of $(m_{i,j})E_{k,0}(n_{i,j}) = (m_{i,k}n_{0,j})$ is $n_{0,0}\begin{pmatrix}
m_{1,k}\\
\vdots\\
m_{n,k}
\end{pmatrix}$, a constant multiple of the $k$th column of $M$.
\end{proof}
Thus, if we know what our automorphism does to the matrices $E_{k,0}$, then we can determine our conjugating matrix (which is only defined up to a scalar multiple).\\

The last thing we need in order to compute the conjugating matrix is the following lemma, which will be helpful in all future computations.
\begin{lemma}\label{ecompute} If $A = (a_{i,j})_{i,j=0}^{n-1}$, then $L^rM^sA = (q^{-2s(i-r)}a_{i-r,j})$ (where the index $i-r$ is taken modulo $n$).
\end{lemma}
\begin{proof}
Note that $L$ is a permutation matrix so that $LA = (a_{i-1,j})$ where the index is taken modulo $n$. It is also clear that $MA = (q^{-2i}a_{i,j})$ since $M$ is diagonal. More generally, this means $L^rA = (a_{i-r,j})$ and $M^sA = (q^{-2is}a_{i,j})$ so that $L^rM^sA = (q^{-2s(i-r)}a_{i-r,j})$.
\end{proof}

\begin{theorem}\label{conjmat}
The conjugating matrix associated with $B$ is $$C = (c_{i,j})_{i,j=0}^{n-1} = (q^{-b^{-1}d(i-aj)^2 - 2cj(i-aj) - acj^2}).$$
\end{theorem}
\begin{proof}
Because the sum $1 + q + q^2 + \cdots + q^{n-1} = \frac{q^n - 1}{q-1} = 0$, and since $n$ is prime, we have that $\sum_{i=0}^{n-1}M^i = \sum_{i=0}^{n-1}e_{0,i} = nE_{0,0}$, and from here, we can use $L = e_{1,0}$ to get our desired matrices (up to a scalar multiple): $$L^j\sum_{i=0}^{n-1}M^i = e_{j,0}\sum_{i=0}^{n-1}e_{0,i} = nE_{j,0}.$$

Now we can see where our automorphism sends these matrices to determine our conjugating matrix as in Lemma \ref{conjcomp}. The matrix $\begin{pmatrix}
a & b\\
c & d
\end{pmatrix}$ sends $e_{j,0}\sum_{i=0}^{n-1}e_{0,i}$ to $e_{aj,cj}\sum_{i=0}^{n-1}e_{bi,di}$ whose first column (indexed by $j=0$) is the first column of $\sum_{i=0}^{n-1}e_{bi,di}$, call this vector $v = \begin{pmatrix}
v_0\\
\vdots\\
v_{n-1}
\end{pmatrix}$. Since $e_{bi,di} = q^{-bdi^2}L^{bi}M^{di}$, and the first column of $M$ is $\begin{pmatrix}
1\\
0\\
\vdots\\
0
\end{pmatrix}$, we have that $v_{bi} = q^{-bdi^2}$ where the index is taken modulo $n$. Hence, $v_i = q^{-bd(b^{-1}i)^2} = q^{-b^{-1}di^2}$ where $b^{-1}$ is the multiplicative inverse of $b$ in $\mathbb{Z}/n\mathbb{Z}$. In general, the first column of $e_{aj,cj}\sum_{i=0}^{n-1}e_{bi,di}$ is $e_{aj,cj}v = q^{-acj^2}L^{aj}M^{cj}(q^{-b^{-1}di^2})_{i=0}^{n-1}$ from which we can finally write down our desired matrix by using Lemma \ref{ecompute}.
$$q^{-acj^2}\left[ L^{aj}M^{cj}(q^{-b^{-1}di^2})\right] = q^{-acj^2}(q^{-b^{-1}d(i-aj)^2}q^{-2cj(i-aj)}) = (q^{-b^{-1}d(i-aj)^2 - 2cj(i-aj) - acj^2}).$$
\end{proof}

\section{Trace Calculation}\label{trcalc}
We now compute the trace of the matrix found in Section \ref{skno}.
\begin{theorem}
Let $B = \begin{pmatrix}
a & b\\
c & d
\end{pmatrix}$. The matrix $C$ we computed in the last section, associated to $B$ has
\begin{equation*}  \Tr(C) = \left(\frac{K_B}{n}\right)\frac{1 + i^{-n}}{1 + i^{-1}}\sqrt{n} \end{equation*} where $K_B := -(b^{-1}d(1-a)^2 + c(2-a))$ \end{theorem}

\proof Above we see that the elements along the diagonal of our conjugating matrix are of the form $c_{i,i} = q^{-b^{-1}d(i-ai)^2 - 2ci(i-ai) - aci^2} = q^{-i^2(b^{-1}d(1-a)^2 + 2c(1-a) + ac)} = q^{-i^2(b^{-1}d(1-a)^2 + c(2-a))}$. Define $K_B := -(b^{-1}d(1-a)^2 + c(2-a))$. Then we have that $\Tr(C) = \sum_{i=0}^{n-1}q^{i^2K_B}$ and this is just the Gauss sum, $G(K_B,n) = \left(\frac{K_B}{n}\right)G(1,n) = \left(\frac{K_B}{n}\right)\frac{1 + i^{-n}}{1 + i^{-1}}\sqrt{n}$ where $\left(\frac{K_B}{n}\right)$ is the Legendre symbol of $K_B$ with respect to $n$.\qed

\section{Determinant Calculation}\label{detcalc}
We now wish to calculate the determinant of $C$, which would be impractical to find through direct calculation and so I will use the following,
\begin{prop}$CC^* = nI_n$ where $C^*$ is the conjugate transpose of $C$.
\end{prop}
\begin{proof}
Let $v = (q^{-b^{-1}di^2})_{i=0}^{n-1}$ as above. I have shown that the $j$th column of $C$ is $e_{aj,cj}v = (e_{a,c})^jv$, that is, $C = (v\ Av\ A^2v\ \cdots\ A^{n-1}v)$ where $A:=e_{a,c}$. Hence, $$CC^* = (v\ Av\ A^2v\ \cdots\ A^{n-1}v)\begin{pmatrix}
v^*\\
v^*A^*\\
v^*(A^*)^2\\
\vdots\\
v^*(A^*)^{n-1}
\end{pmatrix} = \sum_{k=0}^{n-1}A^kvv^*(A^*)^k.$$
We know that $v^*=(q^{b^{-1}dj^2})$ and so $vv^*$ is the matrix $(q^{b^{-1}d(j^2-i^2)})$. By definition, $A^k = q^{-ack^2}L^{ak}M^{ck}$ so that $A^kvv^* = (q^{-ack^2}q^{-2ck(i-ak)}q^{b^{-1}d(j^2-(i-ak)^2)})$ by Lemma \ref{ecompute}. Finally, consider $(A^*)^k = q^{ack^2}(M^*)^{ck}(L^*)^{ak}$. To compute this, we will need an analagous statement to Lemma \ref{ecompute} but for $L^*$ and $M^*$.
\begin{lemma} If $A = (a_{i,j})_{i,j=0}^{n-1}$, then $A(M^*)^s(L^*)^r = (q^{2s(j-r)}a_{i,j-r})$ (where the index $j-r$ is taken modulo $n$).
\end{lemma}
\begin{proof}
Note $L^*$ is a (column) permutation matrix so that $(a_{i,j})L^* = (a_{i,j-1})$, and note $(a_{i,j})M^* = (q^{2j}a_{i,j})$ so that in general $(a_{i,j})(L^rM^s)^* = (a_{i,j})(M^*)^s(L^*)^r = (q^{2s(j-r)}a_{i,j-r})$.
\end{proof}
 
We can now use this to calculate,
\begin{align*}
A^kvv^*(A^*)^k &= (q^{ack^2}q^{2ck(j-ak)}q^{-ack^2}q^{-2ck(i-ak)}q^{b^{-1}d((j-ak)^2-(i-ak)^2)})\\
&= (q^{2ck((j-ak)-(i-ak)) + b^{-1}d((j-ak)^2 - (i-ak)^2)})\\
&= (q^{2ck((j-ak)-(i-ak)) + b^{-1}d(j^2 - i^2 - 2ak(j-i))})
\end{align*}
Therefore, $CC^* = (\sum_{k=0}^{n-1}q^{2ck((j-ak)-(i-ak)) + b^{-1}d(j^2 - i^2 - 2ak(j-i))})$ and it is clear that if $i=j$, then the above becomes $\sum_{k=0}^{n-1}q^0=n$. If $i\neq j$, then the exponent of $q$ is linear in $k$ and since $n$ is prime and $q$ is a primitive $n$th root of unity, we have that we will get a different root of unity for each $k$ so that we get the sum $\sum_{k=0}^{n-1}q^k = 0$. Therefore, $CC^* = nI_n$.
\end{proof}
Using this fact, we can conclude,
\begin{theorem} $\det(C) = \pm\sqrt{n}^n$. \end{theorem}
\begin{proof}
$\det(CC^*) = \det(C)\det(C^*) = \det(C)^2 = n^n\Rightarrow \det(C) = \pm\sqrt{n}^n$.
\end{proof}

Therefore, $$\frac{\Tr(C)}{\sqrt[n]{\det(C)}} = \pm\left(\frac{K_B}{n}\right)\frac{1+i^{-n}}{1 + i^{-1}}$$
Although this should have to do with the Witten-Reshetikhin-Turaev invariant of the mapping cylinder of the corresponding element of the mapping class group \cite{BWP}. We  inspected  these invariants as computed by Lisa Jeffrey \cite{J}, but do not personally see the connection.

\end{document}